\documentclass{amsart}
\usepackage{amsmath,amssymb}
\usepackage{graphicx,subfigure}
\newtheorem{theorem}{Theorem}[section]

\newtheorem{corollary}[theorem]{Corollary}

\newtheorem{lemma}[theorem]{Lemma}

\theoremstyle{definition}

\theoremstyle{remark}
\newtheorem{remark}[theorem]{Remark}

\numberwithin{equation}{section}

\def\Lin{\mathrm{Lin}}
\def\Bil{\mathrm{Bil}}

\newcommand{\al}{\alpha}
\newcommand{\be}{\beta}
\newcommand{\de}{\delta}
\newcommand{\ep}{\varepsilon}

\newcommand{\ga}{\gamma}

\newcommand{\la}{\lambda}
\newcommand{\om}{\omega}
\newcommand{\si}{\sigma}

\newcommand{\De}{\Delta}

\def\hf{\widehat f}

\newcommand{\tu}{\widetilde{u}}

\newcommand{\tW}{\widetilde{W}}

\def\RR{\mathbb{R}}
\def\BB{\mathbb{B}}
\def\ZZ{\mathbb{Z}}
\def\PP{\mathbb{P}}
\def\TT{\mathbb{T}}
\def\QQ{{\mathbb{Q}}}

\renewcommand\SS{\mathbb{S}}
\newcommand{\cM}{{\mathcal M}}
\newcommand{\cN}{{\mathcal X}}

\newcommand{\cR}{{\mathcal R}}
\newcommand{\cS}{{\mathcal S}}

\newcommand{\cW}{{\mathcal W}}

\newcommand{\pd}{\partial}
\newcommand\minus\backslash
\newcommand{\id}{{\rm id}}

\newcommand\lan\langle
\newcommand\ran\rangle

\newcommand{\e}{{e}}

\DeclareMathOperator\Div{div}

\renewcommand\leq\leqslant
\renewcommand\geq\geqslant
\newlength{\intwidth}

\newcommand\loc{_{\mathrm{loc}}}

 \DeclareMathOperator\curl{curl}

\def\T{_{\TT^3}}

\begin{document}

\title[Vortex reconnection in the Navier--Stokes equations]{Vortex
  reconnection in the three dimensional Navier--Stokes equations}

\author{Alberto Enciso}
\address{Instituto de Ciencias Matem\'aticas, Consejo Superior de
  Investigaciones Cient\'\i ficas, 28049 Madrid, Spain}
\email{aenciso@icmat.es, renato.luca@icmat.es, dperalta@icmat.es}

\author{Renato Luc\`a}

\author{Daniel Peralta-Salas}

%
%
\begin{abstract}
We prove that the vortex structures of solutions to the 3D
Navier--Stokes equations can change their topology without any loss of
regularity. More precisely, we construct smooth high-frequency solutions to
the Navier--Stokes equations where vortex lines and vortex tubes of
arbitrarily complicated topologies are created and destroyed in
arbitrarily small times. This instance of vortex reconnection is
structurally stable and in perfect agreement with the existing
computer simulations and experiments. We also provide a
(non-structurally stable) scenario where the destruction of vortex
structures is instantaneous.
\end{abstract}
\maketitle

\section{Introduction}

A fundamental feature of inviscid incompressible fluids in three dimensions is that
the vorticity is transported along the fluid flow. More precisely,
if~$u(x,t)$ is the velocity field of a fluid satisfying the 3D Euler
equations,
\[
\pd_t u+ (u\cdot\nabla)u=-\nabla P\,,\qquad \Div u=0\qquad u(\cdot,0)=u_0\,,
\]
the vorticity $\om:=\curl u$ is known to evolve according to the
transport equation
\[
\pd_t\om=(\om\cdot\nabla) u-(u\cdot\nabla) \om\,.
\]
This equation ensures that the vorticity at time~$t$ can be written in terms
of the initial vorticity $\om_0$ as
\[
\om(\cdot,t)=\phi_{t*}\,\om_0\,,
\]
that is, as the push-forward of the initial vorticity along the
time~$t$ flow generated by the velocity field. It then follows that, as long as the
solution of the Euler equations does not blow up, there are no changes
in the topology of the vortex structures of the fluid, such as
vortex tubes or vortex lines. Recall that a vortex line at
time~$t$ is an integral curve of the vorticity frozen at time~$t$ and
a vortex tube is a toroidal surface (that is, a smooth embedded torus)
arising as a union of vortex lines.

In presence of viscosity, the vorticity is no longer transported along
the flow because the diffusion gives rise to a different phenomenon
known as vortex reconnection. In short, one says that a
vortex reconnection has occurred at time~$T$ if 
the vortex structures at time~$T$ and at time~0 are not homeomorphic, so
there has been a change of topology. For example, a certain vortex
tube can break and there can appear vortex
tubes or vortex lines that are knotted or linked in a
different way as the initial vorticity.

As discussed e.g.\ in~\cite{Kida,Constantin1} and references therein, there is
overwhelming numerical and physical evidence for vortex
reconnection. Particularly relevant for our purposes are the recent
experimental results presented in~\cite{Irvine,Irvine2}, where the
authors study how vortex lines and tubes of different knotted
topologies reconnect in actual fluids using cleverly designed hydrofoils. It is worth mentioning that
the authors observe that these vortex reconnections can occur in very
small times even for fluids with small viscosity. From the
computational point of view, the reconnection of a trefoil-shaped
vortex tube has been recently studied in detail by Kerr~\cite{Kerr}.

In contrast with the wealth of heuristic, numerical and experimental
results on this subject, a mathematically rigorous scenario of vortex reconnection has
never been constructed so far. As discussed in~\cite{Constantin2}, this is
probably due to the fact that with purely analytical methods it is
difficult to analyze the time evolution of the Navier--Stokes
equations to show that vortex reconnection actually takes place. In
the more complex but similar case of magneto-hydrodynamics, magnetic
reconnection (that is, the breaking and topological rearrangement of
magnetic field lines) is known to occur and has deep physical implications.

Our objetive in this paper is to fill this long-standing gap by
providing  a rigorous mechanism of vortex reconnection
in viscous incompressible fluids. Once one comes up with the
mechanism, it is not hard to see that it is actually quite flexible, so here we will strive to present the mechanism in the simplest,
least technical situation.
We provide a detailed discussion of the role that
every element plays in the proof of this result in
Section~\ref{S.discussion}.  In view of the aforementioned
experimental and numerical results~\cite{Irvine,Irvine2,Kerr}, we will be particularly
interested in proving that vortex
structures of any knot or link type can be spontaneously created or destroyed.

The context in which we carry out the analysis is the 3D~Navier--Stokes equations,
\begin{equation*}
\pd_t u+ (u\cdot\nabla)u-\nu\De u=-\nabla P\,,\qquad \Div u=0\,,\qquad
u(\cdot,0)=u_0\,.
\end{equation*}
We will impose periodic boundary conditions, so the spatial variable
will take values in the torus $\TT^3$ with
$\TT:=\RR/2\pi\ZZ$. Hereafter the viscosity~$\nu$ will be a fixed
positive constant. 

We will next state two results on rigorous vortex reconnection. The
first one says that there can be vortex reconnection at arbitrarily
small times. More generally, we will construct a finite
cascade of reconnections at any sequence of times 
\[
T_1<T_2<\cdots <T_n\,,
\]
meaning that there is a smooth solution to the Navier--Stokes
equations, which one can even assume to be global, such that it has
some vortex structures at time~$T_k$ (for each odd integer~$k$) that do not have the
same topology as any of the vortex structures present at the
times~$T_{k-1}$ or~$T_{k+1}$. Furthermore, the scenario of reconnection that we present is
{\em structurally stable}, by which we mean that:
\begin{enumerate}
\item The vortex reconnection phenomenon occurs (with vortex
  structures of the same topology) for any initial
  datum that is close enough in $C^{4,\al}(\TT^3)$ to the initial
  velocity discussed in the theorem. 
\item The existence of non-homeomorphic vortex structures occurs not only between the times
  $T_k$ and~$T_{k\pm1}$ with $k$ odd, but also between any nonnegative times $t_k$ and $t_{k\pm1}$ for which $|T_k-t_k|+
  |T_{k\pm1}-t_{k\pm1}|$ is small enough.
\end{enumerate}
Notice that this condition ensures that the vortex
reconnection is experimentally observable. The result can be stated as
follows, where, for simplicity, when the toroidal
surface encloses a bounded domain (and this condition is non-trivial
when periodic boundary conditions are considered in the equations),
with some abuse of notation we will also refer to this domain as a
vortex tube.

\begin{theorem}\label{T1}
  Given any constants $0=:T_0< T_1<\cdots< T_n$ and $M>0$, for each
  odd integer~$k$ in $[1,n]$ let us denote by $\cS_k$ any finite
  collection of closed curves and
  toroidal domains (with pairwise disjoint closures but possibly
  knotted and linked) that are contained in the unit ball of~$\TT^3$. Then there is a global smooth solution
  $u:\TT^3\times [0,\infty)\to\RR^3$ of the Navier--Stokes
  equations, with a high-frequency initial datum of norm
  $\|u_0\|_{L^2}=M$ and of zero mean, which, for each odd integer
  $k\in[1,n]$, exhibits at time~$T_k$ a set of
  vortex lines and vortex tubes diffeomorphic to~$\cS_k$ that is not
  homeomorphic to any of the vortex structures of the fluid at time~$T_{k-1}$ or~$T_{k+1}$. This
  scenario of vortex reconnection is structurally stable.
\end{theorem}

\begin{remark}
A more visual description of the reconnection process is the
following. For a suitably chosen smooth but highly oscillatory initial
datum~$u_0$, all the vortex structures of the corresponding solution
at times $T_k$ with $k$ even wind around a direction of the torus,
while for times $T_k$ with $k$ odd the solution presents a set of
vortex structures $\cS_k$ of arbitrarily complicated knot types that is
contained in a small ball. Hence the vortex structures $\cS_k$ must
have been created at some time between $T_{k-1}$ and $T_k$ and are
destroyed between $T_k$ and $T_{k+1}$. This phenomenon is still observable
both if one introduces small perturbations of the initial datum and if
one measures the vortex structures at slightly different times. The
roles of even and odd times can obviously be exchanged.
\end{remark}

\begin{remark}
It is possible to prescribe the Reynolds number of the initial datum
instead of its $L^2$~norm. Details are given in Remark~\ref{R.Reynolds}.
\end{remark}

The second result shows that when one starts with an initial
vorticity that is not structurally stable, vortex reconnection can
take place instantaneously. The idea is that one can show that a vortex
tube present at time~0 does not need to survive for positive times, as
the vortex lines sitting on that vortex tube can rearrange
instantaneously and change their topology:

\begin{theorem}\label{T2}
  Given any $M>0$, there is a global $C^\infty$ solution of the
  Navier--Stokes equations $u:\TT^3\times [0,\infty)\to\RR^3$, 
  with initial datum of norm $\|u_0\|_{L^2}=M$ and of
  zero mean, which has a vortex tube at time~$0$ that breaks
  instantaneously. 
\end{theorem}

\begin{remark}
More visually, the proof of the theorem shows that at time~$0$ the space~$\TT^3$ is
  covered by vortex tubes that wind around two of the directions of
  the torus and all the vortex lines are periodic or
  quasi-periodic and tangent to these tori. In contrast, at any small
  enough positive time one (or any finite number) of the initial invariant tori is broken, and
  the rearrangement of the associated initial vortex lines gives
  rise to vortex lines that are not periodic or quasiperiodic and are
  not tangent to a vortex tube. The destruction of this invariant torus
  creates isolated periodic vortex lines with intersecting stable
  and unstable manifolds.
\end{remark}

Let us give some heuristic ideas about the proof of these results. The
first key observation is that, in order to prove these results, the
real enemy is not only the fact that the Navier--Stokes equations are
notoriously difficult to analyze, but rather the need to prove that a
certain vortex structure originating at time $T$ is not diffeomorphic
to any of the structures initially present in the fluid. The
difficulty here is that, as one needs to consider all diffeomorphisms
(not just the flow of the velocity field), the way the diffeomorphisms
can transform the vortex structures is unpredictable. For example, the
diffeomorphism could map a certain vortex line into a curve of length
$10^{-80}$, which one cannot hope to control in a computer-assisted
proof. This difficulty is not merely a mathematical oddity but a
fundamental problem, as it is well known that integral curves and
invariant tori of complicated topology and
arbitrarily small size can bifurcate from vector fields with an extremely
simple structure.

The proof of Theorem~\ref{T1} hinges on choosing an initial datum that
is the sum of several smooth but highly oscillatory fields~$\cW_k$,
that is
\[
u_0=M\, \cW_0+\de_1\, \cW_1+\cdots + \de_n\,\cW_n\,,
\]
and involves an interplay between the (very large) frequencies of the fields and
their relative sizes that ensures that, at time $T_k$, the vortex
structures of the fluid are somehow related to those of~$\cW_k$. Key
to make this argument work is to find two families of vector fields,
which can be conveniently chosen to be Beltrami fields, with
arbitrarily large frequencies and such that in the first family one can
find vortex structures diffeomorphic to those in $\cS_k$ (so this
family is used to construct $\cW_k$ when $k$ is odd) whereas in the
second family all the vortex structures are non-contractible. An
essential property of these families is that they are ``robustly
non-equivalent'', meaning that any (uniformly) small perturbation of a
member of the first family is not topologically equivalent to a small
perturbation of any member of the second family, and viceversa. This
is proved using suitable estimates for Beltrami fields with sharp
dependence on the frequency and KAM-theoretic ideas. It is worth mentioning
that the frequencies we need to consider in the proof of
Theorem~\ref{T1} are much larger
than~$\nu^{-1/2}$, which explains why there is no hope of promoting
this scenario of vortex reconnection to the vanishing viscosity limit.

For the proof of Theorem~\ref{T2} we start with a well chosen initial
condition such that at time~$0$ the torus $\TT^3$ is covered by a
configuration of vortex tubes that is not structurally stable. We then
resort to Melnikov's theory to show that the evolution given by the
Navier--Stokes equations breaks some of these vortex tubes (given by
resonant invariant tori of the vorticity) instantaneously. This
scenario of vortex reconnection does not survive in the vanishing
viscosity limit either (see Remark~\ref{R.nonlinear}).

In both cases we carry
out the computations for initial data of zero mean and arbitrary $L^2$
norm. The global existence of the solutions follows from a
suitable stability theorem for the Navier--Stokes equation and the
fact that our initial data are small perturbations of Beltrami fields
of high frequency and arbitrarily large norm. 

It is worth stressing that, although we have carefully chosen the
initial data not to introduce any inessential technicalities in the
proofs of these results, the underlying ideas are quite flexible and
can be applied to more
general initial data. For instance, Beltrami fields are extremely useful both to make the proof as simple
as possible and to allow us to efficiently deal with vortex structures
of any topology. However, the only part of the argument
where we would not have been able to do without them (at the expense
of losing generality and making the proof more involved) is to show
the global existence of the solutions, which is not essential
for vortex reconnection.

The paper is organized as follows. Building on previous work of two
of the authors~\cite{torus}, in Section~\ref{S.Beltramis} we construct
high-frequency Beltrami fields on the torus with structurally stable
vortex structures of prescribed topology. A key new feature here is
that we derive fine estimates for the norm of these fields in terms
of their frequency. In Section~\ref{S.stability} we establish a
stability theorem for the Navier--Stokes equations with periodic
boundary conditions which yields some of the estimates needed for
Theorem~\ref{T1}. The proof of Theorem~\ref{T1}, which consists of
three main steps, is presented in Section~\ref{S.reconnection}, where
we also discuss a variant of this result in which we prescribe the
Reynolds number of the initial datum instead of its
$L^2$~norm. Section~\ref{S.destruction} is devoted to the proof of
Theorem~\ref{T2} on the instantaneous destruction of vortex tubes. The
paper concludes with some remarks about the role that certain terms
play in the proofs and minor generalizations, which we present in Section~\ref{S.discussion}.

\section{High-frequency Beltrami fields with vortex structures of complex topology}
\label{S.Beltramis}

Our objective in this section is to provide a result on the existence
of high-frequency Beltrami fields with vortex lines and tubes of
prescribed topology, similar to the one proved in~\cite{torus} for the
torus and the 3-sphere. The technical advantage that this result
offers over~\cite{torus} is that here we will control both the norm of the
Beltrami field and the quantitative stability bounds in terms of the
frequency. This is crucial for the proof of Theorem~\ref{T1}.

Let us begin by recalling that a Beltrami field on $\TT^3$ is an eigenfunction of the curl operator:
\begin{equation}\label{Beltrami}
\curl W= N\, W\,.
\end{equation}
We will restrict our attention to Beltrami fields of nonzero
frequency~$ N$, which are necessarily divergence-free and have zero mean:
\[
\int\T W\, dx=0\,.
\]
It is easy to check that the spectrum of curl consists of the points of
the form $ N=\pm|k|$, where $k\in\ZZ^3$ is a 3-vector of integer
components. The most general Beltrami field of frequency~$ N$ is a vector-valued
trigonometric polynomial of the form
\[
W=\sum_{|k|=\pm N}\Big(b_k \, \cos(k\cdot x)+\frac{b_k\times
  k} N\,\sin(k\cdot x)\Big)\,,
\]
where $b_k\in\RR^3$ are vectors orthogonal to~$k$: $k\cdot
b_{k}=0$.

We recall that a set $\cS$ of vortex lines or vortex tubes is {\em
  structurally stable} if they are preserved under
$C^{4,\al}(U)$-small perturbations of the velocity field
modulo a small diffeomorphism of~$\TT^3$ that is close
to the identity in $C^\al(\TT^3)$ (a brief discussion of the H\"older
norms taken to define stability can be found in Section~\ref{S.discussion}). Here $U$ is any fixed open subset
of~$\TT^3$ that contains~$\cS$. With some abuse of notation, we will
often denote by a {\em tube}\/ a toroidal domain in~$\TT^3$ (or its
boundary, which is an embedded torus), and when we say that two tubes
are disjoint (or that a tube is disjoint from a curve) we mean that
the intersection with the closure of the associated domains is empty.

\begin{theorem}\label{T.previo}
Let $\cS$ be a finite union of closed curves and tubes (with pairwise
disjoint closured, but possibly knotted and linked) in
$\TT^3$ that is contained in the unit ball. Then for any large
enough odd integer $ N$ there exists
a Beltrami field~$W$ satisfying~\eqref{Beltrami} and a diffeomorphism $\Phi$ of
$\TT^3$ such that $\Phi(\cS)$ is a union of vortex lines
and vortex tubes of~$W$. This set is contained in the ball of radius
$1/ N$ and structurally stable in the sense that any 
field $W'$ satisfying
\begin{equation}\label{quantstab}
\frac1N\|\curl W-  \curl W'\|_{C^{3,\al}}<\eta
\end{equation}
has a collection of vortex structures given by $\Phi'(\cS)$, where
$\Phi'$ is a diffeomorphism and $\eta$ is a small $N$-independent constant.
Furthermore, the field~$W$ is bounded as
\[
\frac1{C N}< \|W\|_{L^2}< \frac C{\sqrt N}
\]
with $C$ a constant independent of~$ N$.
\end{theorem}

\begin{proof}
It was proved in~\cite{Acta} that there is a small positive constant
$\la\in (0,1)$ and a Beltrami
field~$ w$ on
$\RR^3$ that satisfies 
\[
\curl  w = \la  w\,,
\]
falls off at infinity as $| w(x)|\leq C/|x|$ and has a
collection $\cS'$ of vortex lines and vortex tubes diffeomorphic to $\cS$
(understood now as a subset of the unit ball $\BB_1$ in~$\RR^3$). The
set $\cS'$ is structurally stable under $C^{4,\al}$-small
perturbations and one can
assume that $\cS'$ is also contained in~$\BB_1$.

In view of the sharp decay of~$w$ at infinity, Herglotz's theorem (see e.g.~\cite[Theorem 7.1.27]{Hormander}) ensures
that $ w$ can be written as
\[
 w(x)=\int_{\SS^2}f(\xi)\, e^{i\la x\cdot\xi}\, d\si(\xi)\,,
\]
where $\SS^2$ is the unit sphere with its canonical measure $d\si$ and
$f$ is a complex-valued function in~$L^2(\SS^2)$. Notice that, as $w$ is real-valued,
one necessarily has that $f(\xi)=\overline{f(-\xi)}$, and moreover
\begin{equation}\label{ixif}
i\xi \times f(\xi)-f(\xi)=0
\end{equation}
in $L^2(\SS^2)$ because $w$ is a Beltrami field.

By density we can take a function $g\in
C^\infty(\SS^2)$ with $\|f-g\|_{L^2(\SS^2)}<\ep$, thereby granting
that the field
\[
w_1(x):=\int_{\SS^2}g(\xi)\, e^{i\la x\cdot\xi}\, d\si(\xi)
\]
satisfies
\[
\|w-w_1\|_{C^0(\RR^3)}<C\ep\,.
\]
Without loss of generality, we can assume that $g$ also satisfies the
condition $g(\xi)=\overline{g(-\xi)}$, which ensures that $w_1$ is real-valued.

For each odd integer $ N$, consider the set
\[
\cN_ N:=\{\xi\in\SS^2\cap \QQ^3:  \text{height}(\xi) = N\}\,,
\]
where we recall that the {\em height}\/ of a rational point $\xi\in\SS^2\cap \QQ^3$ is the
least common denomination of the irreducible fractions defined by the
components of~$\xi$. Notice that if the height of~$\xi$ is~$N$, then
$N\xi\in\ZZ^3$, but that the converse is not necessarily true.

It was  proved in~\cite{Duke} that $\cN_ N$ becomes uniformly
distributed as $ N\to\infty$ through the odd integers, and that the
cardinality of this set satisfies
\begin{equation}\label{boundcN}
\frac N C<|\cN_ N|< C N^2\,.
\end{equation}
As $g$ is smooth, the uniform distribution property implies that, for any
large enough odd integer~$ N$ (depending on~$\ep$), the field
\begin{equation}\label{defw2}
w_2(x):=\frac1{|\cN_ N|}\sum_{\xi\in\cN_ N}g(\xi)\, e^{i\la x\cdot\xi}
\end{equation}
approximates $w_1$ in the ball of radius~2:
\[
\|w_1-w_2\|_{C^0(\BB_2)}<\ep\,.
\]
Notice that $-\xi$ is in $\cN_ N$ whenever $\xi\in\cN_ N$, which
ensures that $w_2$ is real-valued by the symmetry of the function~$g$. 
Since $w$, $w_1$ and $w_2$ satisfy the Helmholtz equation $\De w+\la^2
w=0$, standard elliptic estimates then ensure that
\begin{equation}\label{estimww2}
\|w-w_2\|_{C^{6,\al}(\BB_1)}<C\ep\,.
\end{equation}

Taking into account the expression~\eqref{defw2} for the field~$w_2$,
let us define a real-valued vector field on $\TT^3$ as
\[
\tW(x):=\frac1{|\cN_ N|}\sum_{\xi\in\cN_ N}g(\xi)\, e^{i x\cdot( N\xi)}\,.
\]
To see that the periodic boundary conditions are indeed satisfied,
notice that $ N\xi\in\ZZ^3$ by the definition of the
set~$\cN_ N$. Furthermore, since the uniform distribution of
$\cN_ N$ ensures that
\begin{align*}
\lim_{ N\to\infty}|\cN_ N|\|\tW\|_{L^2}^2=\lim_{ N\to\infty}\frac{(2\pi)^3}{|\cN_ N|}\sum_{\xi\in\cN_ N}|g(\xi)|^2=(2\pi)^3\int_{\SS^2}|g(\xi)|^2\, d\si(\xi)\,,
\end{align*}
the estimate~\eqref{boundcN} implies that
\begin{equation}\label{boundtW}
\frac1{C N}< \|\tW\|_{L^2}< \frac C{\sqrt N}
\end{equation}

It follows from the bound~\eqref{estimww2} that in the
unit ball one has
\begin{equation}\label{C6}
\bigg\| \tW\Big(\frac{\la} N\cdot\Big)-w\bigg\|_{C^{6,\al}(\BB_1)}<C\ep\,.
\end{equation}
Let us now define the vector field on $\TT^3$
\begin{equation*}
W:=\frac{\curl(\curl+ N)}{2 N^2} \tW\,,
\end{equation*}
which is easily shown to be a real Beltrami field on the torus with frequency~$ N$.
Notice that $W$ is then close to $\tW$ for large~$ N$; indeed, since
$\cN_ N$ becomes uniformly distributed, as $ N\to\infty$ through odd
integers one has
\begin{align*}
\lim_{ N\to\infty} |\cN_ N|\|\tW-W\|_{L^2}^2&=\lim_{ N\to\infty}
\frac{(2\pi)^3}{|\cN_ N|}\sum_{\xi\in \cN_ N} \bigg|\frac{i\xi\times
  (i\xi\times g(\xi)+ g(\xi))}{2}-g(\xi)\bigg|^2\\
&=(2\pi)^3\int_{\SS^2}\bigg|\frac{i\xi\times
  (i\xi\times g(\xi)+ g(\xi))}{2}-g(\xi)\bigg|^2\, d\si(\xi)\\
&\leq (2\pi)^3\int_{\SS^2}\bigg|\frac{i\xi\times
  (i\xi\times f(\xi)+ f(\xi))}{2}-f(\xi)\bigg|^2\, d\si(\xi)+
C\ep^2\\
&=C\ep^2\,.
\end{align*}
Here we have used the identity~\eqref{ixif} and the fact that
$\|f-g\|_{L^2(\SS^2)}<\ep$. Notice that, by~\eqref{boundtW}, the above estimate readily implies that
\[
\frac1{C N}< \|W\|_{L^2}< \frac C{\sqrt N}\,,
\]
as well as the bound
\begin{equation*}
\|W-\tW\|_{L^2}<\frac{C\ep}{\sqrt N}\,,
\end{equation*}
where we have used~\eqref{boundcN}. Notice that, by the definition
of~$W$ and the bound~\eqref{C6}, 
\begin{equation}\label{Ww}
\bigg\| W\Big(\frac{\la} N\cdot\Big)-w\bigg\|_{C^{4,\al}(\BB_1)}\leq C \bigg\| \tW\Big(\frac{\la} N\cdot\Big)-w\bigg\|_{C^{6,\al}(\BB_1)}<C\ep\,.
\end{equation}
The first part of the theorem then follows easily from the structural stability of
$\cS'$. 

To complete the proof of the theorem, for convenience let us introduce a variable $y$
that takes values in $\BB_1$ and set $x:=\la y/N$, which will
eventually be interpreted as the original variables on~$\TT^3$. Subscripts~$x$
or~$y$ will denote that the various norms are computed with respect to
that variable and the curl operator will always be defined using the
variable~$x$. 

The structural stability of the set~$\cS'$ of
vortex lines and tubes of~$w$ implies that if
\begin{equation}\label{W'w}
\bigg\|\frac1N \curl W'- w\bigg\|_{C^{3,\al}_y(\BB_1)}<2\eta\,,
\end{equation}
with $\eta$ a certain constant that does not depend on~$N$, then
$\curl W'$ has a set of vortex lines and tubes diffeomorphic to~$\cS$
and contained in the region $y\in\BB_1$. Taking the curl and small~$\ep$, a short
computation shows that the inequality~\eqref{Ww} implies that
\[
\bigg\|\frac1N \curl W- w\bigg\|_{C^{3,\al}_y(\BB_1)}\leq \frac C\la \| W-w\|_{C^{4,\al}_y(\BB_1)}
\]
can be taken smaller than~$\eta$.  Hence the 
inequality~\eqref{W'w} will automatically hold provided that we have chosen~$\ep$
small enough and
\[
\frac1N\|\curl W-\curl W'\|_{C^{3,\al}_y(\BB_1)}<\eta\,.
\]
Using now that
\begin{align*}
\frac1N\|\curl W-\curl
W'\|_{C^{3,\al}_y(\BB_1)}&=\frac1N\sum_{j=0}^3\|\nabla_y^j(\curl W-\curl
W')\|_{C^0_y(\BB_1)}\\
&\qquad \qquad \qquad\qquad+\frac1N [\nabla_y^3(\curl W-\curl
W')]_{\al,y,\BB_1}\\
&=\sum_{j=0}^3\la^jN^{-1-j}\|\nabla_x^j(\curl W-\curl
W')\|_{C^0_x(\BB_{\la/N})}\\
&\qquad \qquad  +\la^{3+\al}N^{-4-\al} [\nabla_x^3(\curl W-\curl
W')]_{\al,x,\BB_{\la/N}}\\
&\leq \frac1N\|\curl W-\curl
W'\|_{C^{3,\al}_x(\BB_{\la/N})}\\
&\leq \frac1N\|\curl W-\curl
W'\|_{C^{3,\al}_x(\TT^3)}
\end{align*}
with $0<\la<1$, where $[\cdot]_{\al,y,\BB_1}$ denotes the H\"older seminorm of
exponent~$\al$ computed with respect to the variable~$y$ in the domain~$\BB_1$, the theorem follows.
\end{proof}

\section{A stability result for the Navier--Stokes equations}
\label{S.stability}

In the proofs of Theorems~\ref{T1} and~\ref{T2} we will need to estimate a solution~$u(x,t)$ to the Navier--Stokes
equations on~$\TT^3$ whose initial datum $u_0$ is a small perturbation
of a Beltrami field~$W$ with frequency~$ N$. Our goal in this section
is to provide a stability result for the Navier--Stokes equations that
is very well suited for this task. 

We will state this stability
result in terms of perturbations of a solution $w(x,t)$ to the
Navier--Stokes equations that is in
$L^2(\RR^+,W^{r,\infty}(\TT^3))$ with $r\geq1$. Notice that, if $W$ is a Beltrami
field with frequency~$N$, then the solution
with initial datum $w(\cdot,0)=W$ is
\[
w(x,t):=e^{-\nu  N^2 t}\, W(x)\,,
\]
so the result obviously applies when the initial datum is a Beltrami
field. 

As a further simplification, we will assume that the initial datum has
zero mean:
\begin{equation*}
\int\T u_0\, dx=0\,.
\end{equation*}
This will always be the case if the initial datum is a linear
combination of Beltrami fields. As the average velocity is a conserved quantity of the
Navier--Stokes equations, it will then follow that
\[
\int\T u(x,t)\, dx=0
\]
for all~$t$.

Let us introduce some notation that we
will use in the rest of the paper. Since we will only need to deal with symmetric
tensors, we shall denote by $\otimes$ the {\em symmetric tensor product}\/,
namely 
\[
(v\otimes w)_{ij} := \frac12 ( v_{i}w_{j} + w_{i}v_{j} )\,.
\]
We will also use the shorthand notation
\[
|\nabla^m w|^2:=\sum_{|\al|=m}|\pd^\al w|^2\,.
\]
In particular, for a time-dependent vector field $w(x,t)$ one can write
\[
\|w\|_{H^r}^2:=\sum_{m=0}^r\int\T |\nabla^m w|^2\, dx\,,\qquad
\|w\|_{L^2W^{r,\infty}}^2:=\sum_{m=0}^r\int_0^\infty\|\nabla^mw(\cdot,t)\|_{L^\infty}^2\,dt\,.
\]

It is worth stressing that, while the existence of stability
theorems for the Navier--Stokes equations is not surprising  (see
e.g.~\cite{Titi} for a similar result ensuring the stability of
solutions on~$\RR^3$ that belong to the space
$L^4(\RR^+,H^1(\RR^3))$, the specific dependence of our bounds
on the various norms of the unperturbed solution is key in the proof
of Theorem~\ref{T1}. In particular, it is essential to ensure that our
bounds do not depend exponentially on higher norms of the unperturbed
solution~$w\in L^2 W^{r,\infty}$, but only on its $L^2L^\infty$ norm. To state the theorem in a form that will be particularly useful later on, we find it convenient to recursively define the
quantities $Q^w_r(t)$ associated with a vector field $w(x,t)$ as
\begin{equation}\label{Def:Recurs}
Q^w_{0}(t) := 1
\qquad
Q^w_{r}(t) := 1+   \sum_{m=0}^{r-1} Q^w_m(t) 
\int_{0}^{t} \| \nabla^{r-m} w (\cdot,\tau) \|^{2}_{L^{\infty}} \, d\tau \,.   
\end{equation}

\begin{theorem}\label{T.stability}
  Given some $r\geq1$ and any $\si<1$, let $w$ be a global solution to the
  Navier--Stokes equations in $L^2(\RR^+,W^{r,\infty}(\TT^3))$, with initial datum $w_0:=w(\cdot,0)$ of
  zero mean. Then there is a positive
constant $C$, which does not depend on~$w$, such that for any divergence-free 
initial datum $u_0$ with zero mean and
\begin{equation}\label{boundu0w0}
\|u_0-w_0\|_{H^r}<\frac1C\, Q_r^w(\infty)^{-\frac12}\, e^{-C\|w\|_{L^2L^\infty}^2}\,,
\end{equation}
the corresponding solution $u(x,t)$ to the Navier--Stokes equations is
global and satisfies
\[
\|u(\cdot,t)-w(\cdot,t)\|_{H^m}\leq C\,
Q_m^w(t)^{\frac12}\,
e^{C\|w\|_{L^2L^\infty}^2}\|u_0-w_0\|_{H^m}\, e^{-\nu\si t}
\]
for all $0\leq m\leq r$ and $t>0$, with a $\si$-dependent constant.
\end{theorem}

\begin{proof}
Denoting by $P_w$ the pressure function of $w$, it is readily checked that $u$ is a solution of the
Navier--Stokes equations with data $u_{0}$ and pressure $P$ 
if and only if the difference $v:=u-w$ satisfies the equation
\begin{equation}\label{eqv}
\partial_{t} v + \Div( v \otimes  v + 
 2v \otimes w)
-   \nu \Delta  v  =  -  \nabla P_{ v}   \,,\quad \Div  v  =  0\,,\quad  v(\cdot,0) =   v_{0}\,.
\end{equation}
Here $P_{ v} := P -P_w$ and $v_0:=u_0-w_0$. It is
standard that for any $ v_{0} \in H^{r}$ with $r \geq 1$
there exists a local in time solution $v \in
L\loc^2([0,T),H^{r+1}(\TT^3))$, which is a continuous function of time
with values in $H^r(\TT^3)$, provided that $T$ is small enough. Our goal is to show that the
solution is actually global, and this will follow provided that we
show that one can control the $H^r$~norm of~$v$.

To this end we will consider the evolution in time of the energies
\begin{equation}\label{defhm}
h_m(t):=\sum_{j=0}^m\int\T |\nabla^j v(x,t)|^2\, dx
\end{equation}
with $m\leq r$ and show that these quantities do not blow up for any
finite~$t$ (and, in particular, are bounded at time~$T$). In the case $m=0$, it is enough to multiply Equation~\eqref{eqv} by $v$ and integrate by parts to obtain
\begin{align*}
\frac12\frac {dh_0}{dt}&=-\nu\int\T|\nabla v|^2\,dx-\int\T v_i\, v_j\,
                         \pd_jw_i\, dx\\
&=-\nu\int\T|\nabla v|^2\,dx+\int\T w_i \, v_j\, \pd_jv_i\, dx\\
&\leq -\nu\int\T|\nabla v|^2\,dx+\|w\|_{L^\infty}\|v\|_{L^2}\|\nabla
  v\|_{L^2}\\
&\leq -\nu\si \|\nabla v\|_{L^2}^2+C\|w\|_{L^\infty}^2 h_0\\
&\leq (-\nu\si+C\|w\|_{L^\infty}^2)h_0 \,,
\end{align*}
which yields
\begin{equation}\label{h0}
h_0(t)\leq \|v_0\|_{L^2}^2\, e^{-2\nu \si
  t+C\int_0^t\|w(\cdot,\tau)\|_{L^\infty}^2\, d\tau}\,.
\end{equation}
Here $\si$ is any fixed number in the interval $(0,1)$ and we have used that, as the mean of~$v$ is a conserved quantity for
Equation~\eqref{eqv}, $v(\cdot,t)$ has zero
mean for all~$t$, so the Poincar\'e inequality on the torus ensures
\begin{align*}
\|\nabla v\|_{L^2}\geq \|v\|_{L^2}\,.
\end{align*}

The estimate for $h_m$ with~$m\geq 1$ can be proved by
induction.  In view of the estimate~\eqref{h0}, let us make the induction
hypothesis that for any $0\leq m\leq r-1$ the function $h_m$ is
bounded as
\begin{equation}\label{hm}
h_m(t)\leq 
C \,
Q^w_{m}(t) \, e^{C \int_{0}^{t}\|w(\cdot,\tau)\|_{L^\infty}^2 d\tau }   \|v_0\|^{2}_{H^m}\, e^{-2\nu \si t} \,.
\end{equation}
It is now enough to prove the bound for $h_r$. For this, let us commute the
equation for~$v$ with the spatial derivative $ \pd^\al $, with a
multiindex of order $|\al|\leq r$, to find
\begin{equation*}
\pd_t\pd^\al v-\nu\De\pd^\al v+\sum_{\be\leq\al} \binom\al\be 
   \Big[
 \big( ( \partial^{\beta}  v + \partial^{\beta} w) \cdot \nabla \big) \partial^{\alpha - \beta}  v 
 + 
 ( \partial^{\beta}  v \cdot \nabla ) \partial^{\alpha - \beta} w  
 \Big]
   = - \nabla \partial^{\alpha} P_v \, ,
\end{equation*}
where as is customary the condition $\beta \leq \alpha$ and the
combinatorial numbers should be understood
componentwise. Multiplying this equation by $ \pd^\al v$, integrating
by parts and using the Cauchy--Schwartz inequality and the fact that
$|\al|\leq r$
one then infers
\begin{multline}\nonumber
\frac12\frac d{dt}\int_{\TT^3}  |\partial^{\alpha}  v|^{2} +  \nu   \int_{\TT^{3}}  | \nabla \partial^{\alpha} v|^{2}
\leq 
   \ep  \int_{\TT^{3}}  | \nabla \partial^{\alpha} v |^{2}
      +  C 
 \int_{\TT^{3}}
  |w |^{2}  |\pd^{\alpha} v |^{2}
  \\
+ C 
   \sum_{m=0}^{r-1} 
  \int_{\TT^{3}}
  |\nabla^{m} v |^{2}  |\nabla^{r -m} v |^{2}
   +   
  C   \sum_{m=1}^r
 \int_{\TT^{3}}
  |\nabla^{m} w |^{2}  |\nabla^{r -m} v |^{2}
 \end{multline}
where $\ep$ is a small positive constant and the constant~$C$ 
depends on~$\ep$, $\nu$ and~$r$.

Using now the fact that $v(\cdot,t)$ has zero
mean for all~$t$ and the
Gagliardo--Nirenberg inequality for zero-mean fields
\[
\|\nabla^m v\|_{L^\infty}\leq C\|\nabla^{m+2}
v\|_{L^2}^{1/2}\|\nabla^m v\|_{L^6}^{1/2}\leq C\|\nabla v\|_{H^{m+1}}^{1/2}\|v\|_{H^{m +1}}^{1/2}
\]
for $0\leq m\leq r-1$, we obtain, recalling the
definition~\eqref{defhm} of the energies~$h_m$, that
\begin{align*}
 \frac12\frac d{dt}  \int_{\TT^3}  |\partial^{\alpha}  v|^{2}  &+ (\nu -\ep)  \int_{\TT^3}  | \nabla \partial^{\alpha}  v|^{2}  
\leq  
C \|\nabla v\|_{H^{r}}h_{r}^{3/2}
+
C \|w\|^{2}_{L^{\infty}} h_{r}
\\
&\qquad \qquad \qquad \qquad \qquad \qquad\qquad+
C   \sum_{m=1}^r   \| \nabla^m w \|^{2}_{L^{\infty}}   \|v \|_{H^{r - m}}^{2}
\\ &
\leq 
\ep\|\nabla v\|_{H^{r}}^2+C h_{r}^{3}+
C \|w\|^{2}_{L^{\infty}} h_{r}
+C    \sum_{m=1}^r  \| \nabla^m w  \|^{2}_{L^{\infty}}  h_{r - m}\,.
\end{align*}
These inequalities are true for all times up to the maximal time of
existence of the solution. Summing them over all multiindices $\alpha$ such that $|\al|\leq r$
and using that
\[
\|\nabla v\|_{H^r}\geq \|v\|_{H^r}
\]
by the Poincar\'e inequality on the torus (again exploiting that $v(\cdot,t)$ has zero
mean), one then infers
\begin{align*}
\frac{d h_r}{dt}&\leq -2( \nu -c \ep )\|\nabla v\|^{2}_{H^r}+C h_r^3+ C \|w\|^{2}_{L^{\infty}} h_r 
+C  \sum_{m=0}^{r-1}  \| \nabla^{r-m} w  \|^{2}_{L^{\infty}}  h_m
\\
&\leq -2(\nu-c \ep)h_r+C h_r^3+ C \|w\|^{2}_{L^{\infty}} h_r
+C   \sum_{m=0}^{r-1}  \| \nabla^{r-m} w  \|^{2}_{L^{\infty}}  h_m\,,
\end{align*}
where the constant $c$ depends on~$r$ but not on~$\ep$.

As long as 
\begin{equation}\label{IsSat}
h_r^{2}(t) \leq \de
\end{equation}
for some small constant $\de$ that depends on~$\ep$ and~$\nu$ one clearly has
\[
-2(\nu-c \ep)h_r(t)+C h_r^3(t)
\leq -2\nu\si h_r(t) \,,
\]
Hence, recalling the induction hypothesis (\ref{hm}) we deduce that as
long as~\eqref{IsSat} holds one has
\begin{multline*}
\frac{d h_r}{dt} 
\leq 
   -2\nu \si h_r  + C \|w\|^{2}_{L^{\infty}} h_r +
\\  
 +
Ce^{-2\nu\si t  + C \int_0^t\|w(\cdot,\tau)\|_{L^{\infty}}^2  d\tau} \,
\sum_{m=0 }^{r-1} \| \nabla^{r-m} w  \|^{2}_{L^{\infty}} 
     Q^w_{m}(t) h_m(0) \,.
\end{multline*}
Thus using the Gr\"onwall inequality and recalling the definition
given in~\eqref{Def:Recurs} of $Q^w_r(t)$ (which implies, in
particular, that these quantities are non-decreasing), we arrive at
\begin{align*}
h_r(t) 
&
\leq 
Ce^{-2 \nu\sigma t  + C \int_0^t\|w(\cdot,\tau)\|_{L^{\infty}}^2 d\tau} \,
\!\Bigg( \! 1 \! +
  \! \sum_{m=0}^{r-1 } 
   Q_m^w(t) \!  \int_{0}^{t}  \| \nabla^{r-m} w(\cdot,\tau)  \|^{2}_{L^{\infty}}  d\tau 
      \Bigg)  h_{r}(0)  
\\ \nonumber
&
=C
e^{-2 \nu\sigma t  + C \int_0^t\|w(\cdot,\tau)\|_{L^{\infty}}^2 d\tau} \,
   Q^w_{r}(t) \,  h_{r}(0)      \,.
\end{align*}
In particular, the smallness assumption (\ref{IsSat}) is satisfied
provided that the inequality~\eqref{boundu0w0} is satisfied for some large
enough, $\ep$-dependent constant~$C$. Hence the solution exists for
all positive times and the theorem follows from
the bound for $h_r(t)$.
\end{proof}

We shall next state as a corollary a concrete instance of the theorem
that is precisely what we will need to apply later on. For the ease of
notation, here and in what follows let us agree to say that two
quantities $q$ and~$q'$ satisfy the condition
\[
q\ll q'
\]
if $q< \frac1Cq'$ for some large but fixed constant that does not depend on any
relevant parameters. With this notation, the application of
Theorem~\ref{T.stability} that we will actually use is as follows:

\begin{corollary}\label{C.stability}
Given $r\geq1$ and any $\si\in[0,1)$, let $w(x,t)$ satisfy the hypotheses of Theorem~\ref{T.stability} with
\[
\|w\|_{L^2W^{m,\infty}}< C(1+ N^{m-1})\,,
\]
where $ N$ is a large constant, $0\leq m\leq r$ and $C$ is a constant that does not depend
on~$ N$. Then if
\[
\|u_0-w_0\|_{H^r}\ll  N^{1-r}\,,
\]
then the solution to the Navier--Stokes equations with initial datum
$u_0$ is globally defined and satisfies, for all $0\leq m\leq r$,
\[
\|u(\cdot,t)-w(\cdot,t)\|_{H^m}\leq C( N^{m-1}+1)\e^{-\nu\si t}\|u_0-w_0\|_{H^m}\,.
\]
\end{corollary}

\begin{proof}
It suffices to note that $Q_m^w(t)$ is a non-decreasing function
of~$t$ and that the assumption on $\|w\|_{L^2W^{m,\infty}}$
implies that
\[
Q_m^w(\infty)< C(1+ N^{2m-2})
\]
for all $0\leq m\leq r$.
\end{proof}

\section{Reconnection of vortex tubes}
\label{S.reconnection}

Our goal in this section is to prove Theorem~\ref{T1}, thereby establishing that one can choose a smooth
initial datum of arbitrary $L^2$ norm such that the corresponding
solution to the Navier--Stokes equations features $n$~vortex
reconnection processes at arbitrarily small times $T_1<\cdots
<T_n$. As we will see, the vortex structures that are created and
destroyed can have arbitrarily complicated topologies. 

For the sake of
clarity, we will divide the proof in three steps, where we will
construct a collection of Beltrami fields of arbitrarily high
frequencies that are stably
non-equivalent (Step~1), derive uniform estimates for the linearization of the
Navier--Stokes equations around a high-frequency Beltrami field (Step~2), and
make a clever choice of some free constants that will allow us to
derive the desired result (Step~3). As we will discuss in
Section~\ref{S.discussion}, the mechanism of vortex reconnection that
we have presented here is quite flexible. What is key, however, it to
choose carefully the frequencies and relative sizes of the various
terms that we will use to construct the initial datum: in a way, the
heart of the matter is a delicate interplay between several vector fields, all of
which are of high frequency, combined with a uniform topological
non-equivalence result for certain of these fields.

\subsubsection*{Step 1: Beltrami fields that are stably
  topologically non-equivalent}


In the proof of the theorem we will need to consider two different
families of high-frequency Beltrami fields. The first family of
Beltrami fields, which is not explicit but features robust contractible vortex structures
of complicated topology, is obtained by repeatedly
using Theorem~\ref{T.previo}:



\begin{lemma}\label{L.isotopy}
For every odd integer $1\leq k\leq n$, let $\cS_k$ be any finite collection
of (pairwise disjoint but possibly knotted and linked) closed curves and tubes that is contained in
the unit ball. Then for any large enough odd integers $ N_k$ there
are Beltrami fields $W_k$ on~$\TT^3$ with the following properties:
\begin{enumerate}
\item $\curl W_k= N_k\, W_k$.
\item $W_k$ has a collection of vortex lines and vortex tubes that is
  diffeomorphic to $\cS_k$, structurally stable (in the sense of~\eqref{quantstab}) and contained in the
  ball $\BB_{1/ N_k}$ of radius $1/ N_k$.
\item For any nonnegative integer $m$, the $H^m$ norm of $W_k$ satisfies
\[
\frac{ N_k^{m-1}}C<\|W_k\|_{H^m}<C N_k^{m-\frac12}
\]
with a constant $C$ that depends on~$m$ but not on~$ N_k$.
\end{enumerate}
\end{lemma}

\begin{proof}
This immediately follows by applying Theorem~\ref{T.previo} to the
sets $\cS_k$. Although in this theorem we had only stated $L^2$
bounds, the $H^m$ bound is immediate because for a Beltrami field the
norm $\|W_k\|_{H^m} $ is obviously equivalent to $N_k^m\, \|W_k\|_{L^2}$.
\end{proof}

The second family of Beltrami fields that we need to consider is given by the fields
\[
B_ N:=(2\pi)^{-3/2}\, (\sin  N x_3,\cos  N x_3,0)\,,
\]
where $ N$ is a positive integer. Notice that $B_ N$ satisfies the equation
\[
\curl B_ N= N\,B_ N
\]
and has been normalized so that $\|B_ N\|_{L^2}=1$.

It is not hard to see that, for any integer~$N$ and any odd integer
$1\leq k\leq n$, the Beltrami fields $W_k$ and $B_N$ are not
topologically equivalent, meaning that there are vortex structures in
$W_k$ that are not homeomorphic to any of the vortex
structures of~$B_N$. The idea here is that all the
vortex structures of $B_N$ are non-contractible, while $W_k$ has a set
of vortex structures diffeomorphic to~$\cS_k$, which is, in
particular, contractible.

The central result of Step~1 is to show that this situation is robust,
meaning that the same property is true for any suitably small
perturbations of these fields. While this can seem pretty obvious at
first sight, the proof is not trivial, as it employs in a key way that
the integral curves of a uniformly (with respect to~$N$) small
perturbation of $B_N$ are confined in narrow regions of $\TT^3$ for
all times. Without this bound, for very large times the perturbation
could get the perturbed integral curve far from the region where the
integral curve of $B_N$ lies and this could translate into the
perturbed field actually having contractible integral curves
diffeomorphic to~$\cS_k$ (for example, the integral curves might wind
around a direction for a long time but then unwind until 
eventually becoming a complicated closed contractible curve). The key bound that prevents this from
happening is obtained from a KAM-type argument applied to a zero-mean
divergence-free field that lives on a three-dimensional space.

We shall next state the ``robust non-equivalence'' result that we will
need to prove the existence of vortex reconnection. To simplify the
statements, with a slight abuse of notation, by a vortex line and a vortex tube
of a general divergence-free field we will respectively mean an integral curve
and a (domain bounded by an) invariant torus of its curl. We also recall that a vortex line or a vortex tube in~$\TT^3$ is {\em contractible}\/
if and only if it is homeomorphic to a curve or tube contained in the
unit ball.

\begin{lemma}\label{L.contractible}
For any positive integer $ N$ and any odd integer $1\leq
k\leq n$, suppose that $W'$ and $B'$ are any vector
fields on $\TT^3$ with 
\[
\|W_k-W'\|_{H^r}+\|B_ N-B'\|_{H^r}\ll  1
\]
for some $r\geq7$, where the implicit constant in this inequality
does not depend on~$N$ or $N_k$. Then: 
\begin{enumerate}
\item $W'$ has a collection of vortex lines and vortex tubes
diffeomorphic to $\cS_k$. 
\item $B'$ does not have any contractible vortex lines or
vortex tubes.
\end{enumerate}
\end{lemma}
\begin{proof}
In fact, and using that $\curl B_N=N B_N$ and $\curl W_k=N_k W_k$, we will prove the result
under the weaker hypothesis that 
\[
\bigg\|
 W_k-\frac1{N_k}\curl W'\bigg\|_{C^{3,\al}}
+ \bigg\|B_ N-\frac1N\curl B'\bigg\|_{C^{3,\al}}\ll1\,.
\]
The proofs of the statements for~$B'$ and $W'$ are logically
independent. The case of~$W'$ can be essentially
read off Theorem~\ref{T.previo} (and Equation~\eqref{quantstab}) using that for $r\geq7$ the $H^r$ norm
controls the $C^{4,\al}$ norm. Specifically, notice that
\begin{multline*}
\bigg\|
 W_k-\frac1{N_k}\curl W'\bigg\|_{C^{3,\al}}=\frac1{N_k}\|\curl
 W_k-\curl W'\|_{C^{3,\al}}\\
\leq \frac C{N_k}\|W_k-W'\|_{C^{4,\al}}\leq
 \frac C{N_k}\|W_k-W'\|_{H^r}\,.
\end{multline*}
Since $N_k\geq1$, the factor $1/N_k$ means that one can even take
$\|W_k-W'\|_{H^r}<\eta N_k$, but we will not need this improvement. 

Let us now focus on the statement for
the field~$B'$. Up to reparametrization, the integral curves of $\curl B'$ coincide
with those of
\[
B'':=\frac 1N\curl B'= B_N+b\,,
\]
with $b:=\frac1N\curl B'-B_N$, so it suffices to prove the result
for~$B''$.  The claim for $B'$ hinges on a uniform KAM theoretic estimate. Essentially, what this
KAM estimate gives us is that if $x(s)$ is an integral curve of~$B''$
whose initial condition $x^0=(x_1^0,x_2^0,x_3^0)$ is such that
$N\, x_3^0\in I_j^n(\frac1{10})$ and $\|b\|_{C^{3,\al}}$ is smaller
than some $N$-independent constant, then $N\, x_3(s)\in I_j^n(\frac15)$
for all times $s$. Here we are using the notation $I_j^n(\de)$ for the
(non-pairwise disjoint)
intervals in the circle $\RR/2\pi N\ZZ$ defined as
\begin{align*}
I_1^n(\de)&:=\Big\{z: z- 2\pi n\in \Big (\!-\frac\pi 4-\de,
\frac\pi 4+\de \Big)\mod 2\pi N \Big\}\,,\\
I_2^n(\de)&:=\Big\{z: z- 2\pi n\in \Big (\frac\pi 4-\de,
\frac{3\pi} 4+\de \Big)\mod 2\pi N \Big\}\,,\\
I_3^n(\de)&:=\Big\{z: z- 2\pi n\in \Big (\frac{3\pi} 4-\de,
\frac{5\pi} 4+\de \Big)\mod 2\pi N \Big\}\,,\\
I_4^n(\de)&:=\Big\{z: z- 2\pi n\in \Big (\frac{5\pi} 4-\de,
\frac{7\pi} 4+\de \Big)\mod 2\pi N \Big\}\,,
\end{align*}
where $0\leq n\leq N-1$ is an integer.
Let us complete the proof of the lemma under the assumption that this
result is true. The proof will be presented later.

Let $x(s)$ be an arbitrary (possibly non-periodic) integral curve of
$B''$. Since for all $s$ the third coordinate is such that $N\,x_3(s)$ lies in an
interval $I_j^n(\frac15)$, it follows that the $i^{\mathrm{th}}$ component of the
field $B_N$ (where $i=1$ if $j$ is even and $i=2$ is $j$ is odd)
evaluated on this integral curve satisfies
\[
|B_{N,i}(x(s))|>c
\]
for all $s$ and some positive constant $c$ that does not depend
on~$N$. This is simply because, on each of these intervals, either the
sine or the cosine are bounded away from~$0$. Notice that obviously
\[
|B_i''(x(s))|> c-\|b\|_{C^0}>\frac c2
\]
provided that $\|b\|_{C^0}$ is small enough. It then follows that, for
the integral curve that we are considering, $x_i(s)$ defines a
map $\RR\to\TT$ whose derivative is bounded away from zero. Hence the
map winds around the circle~$\TT$ either in the positive or the
negative direction for all times, so, in particular, it is not
homotopic to a constant map $\RR\to\TT$. Hence the integral curve
$x(s)$ must be homotopically nontrivial, so it cannot be contractible.

It only remains to prove the auxiliary technical result that an integral curve of $B''$ whose
initial datum has $Nx_3^0$ in $I_j^n(\frac1{10})$, then $N\, x_3(s)\in I_j^n(\frac15)$ for all
times provided that $\|b\|_{C^{3,\al}}$ is small enough. This is
easier to obtain in the rescaled variables $y:=Nx$ and $\si:=(2\pi)^{-3/2}Ns$. To this end, let us start by considering the integral curve~$y(\si)$
of $B_N$ with initial condition $y^0 =(y_1^0,y_2^0,y_3^0)$. Since the
ODE reads as
\[
\frac{dy_1}{d\si}=\sin y_3\,,\qquad \frac{dy_2}{d\si}=\cos y_3\,,\qquad \frac{dy_3}{d\si}=0\,,
\]
the integral curves of $B_N$ are explicitly given by
\[
y_1(\si)=y_1^0+\si\, \sin y_3^0\,,\qquad y_2(\si)=y_2^0+\si\, \cos y_3^0\,,\qquad y_3(\si)=y_3^0\,.
\]
Notice that $y_j$ takes values in $\RR/2\pi N\ZZ$ and that the
regions in $\TT^3$ defined by $\{y:Ny_3\in I^n_j(\frac15)\}$ are invariant
under the flow of the ODE.

It is easy to see that the field $B_N$ satisfies non-degeneracy KAM
conditions in the sense that the slope of the integral curves on each
torus $\{ y: y_3=y_3^0\}$ varies from torus to torus (that is, it is
not a constant function of~$y_3^0$). Indeed, this
slope is given by 
\[
\frac{dy_1/d\si}{dy_2/d\si}=\tan y_3^0
\]
when $y_3^0\in I^n_j(\frac15)$ with $j$ odd and by
\[
\frac{dy_2/d\si}{dy_1/d\si}=\cot y_3^0
\]
when $j$ is even. Since the derivative of the tangent or cotangent
never vanishes (and the function is everywhere defined on these
intervals), the twist condition of the KAM theorem is automatically
satisfied. Hence~\cite{KKP} any divergence-free field $B''=B_N+b$ of
zero mean on $\TT^3$ with $b$ small enough in $C^{3,\al}$ has the
following property: for any $\be\in\TT$, the field $B''$ has an
invariant torus given by 
\[
\Phi(\{x\in\TT^3: x_3=\be\})\,,
\]
where $\Phi$ is a $C^{3,\al}$ diffeomorphism with
$\|\Phi-\id\|_{C^\al}<C\|b\|_{C^{3,\al}}^{1/2}$. Since the invariant
tori have codimension~1, there can be no diffusion, which ensures that
for any integral curve of~$B''$,
\[
|y_3(\si)-y_e^0|=N|x_3(s)-x_3^0|<C\|b\|_{C^{3,\al}}^{1/2}
\]
for all times~$s$. In particular, if the initial datum has $Nx_3^0$ in
$I_j^n(\frac1{10})$, then $N\, x_3(s)\in I_j^n(\frac15)$ for all
times provided that $\|b\|_{C^{3,\al}}\ll1$. The lemma then follows
upon noticing that
\[
\|b\|_{C^{3,\al}}= \frac1N\|\curl B_N-\curl B'\|_{C^{3,\al}}\leq \frac
1N\|B_N-B'\|_{C^{4,\al}}\leq \frac CN\|B_N-B'\|_{H^r}
\]
and that the field $b$ is obviously divergence-free and of
zero mean because it is the curl of another field.
\end{proof}


\subsubsection*{Step 2: Estimates for the solution}

We shall next use the families of Beltrami fields $B_N$, $W_k$ to
construct a suitable initial datum of norm~$M$ for the Navier--Stokes equations
which exhibits vortex reconnection at times $T_1,\dots, T_n$.

For this, let us consider large nonnegative integers $ N_0,\dots, N_n$ with
\[
1\ll N_n\ll N_{n-1}\ll\cdots \ll N_2 \ll \sqrt{N_1}\,,
\]
and small positive real numbers
$\de_1,\dots,\de_n$ with $\de_{j+1}\ll \de_j$. 
Take the solution to the Navier--Stokes equations with
initial condition
\[
u_0:= M\, \cW_0+\sum_{j=1}^n\de_j\, \cW_j\,,
\]
where we are using the notation
\[
\cW_j:=\begin{cases}
B_{ N_j} &\text{if $j$ is even,}\\
W_j &\text{if $j$ is odd}
\end{cases}
\]
and the fields $B_{ N}$ and $W_j$ were defined in Step~1.

Let us define the function
\[
w(x,t):=M\, \cW_0(x)\, e^{-\nu  N_0^2 t}\,,
\]
which satisfies the Navier--Stokes equations on $\TT^3\times\RR^+$ and
the hypotheses of Corollary~\ref{C.stability}. By the
assumptions on $\de_j, N_j$ and the bounds for
$W_j$ stated in Lemma~\ref{L.isotopy}, the field
\[
v_0:=\sum_{j=1}^n\de_j\, \cW_j
\]
is bounded as
\begin{equation}\label{v0}
\|v_0\|_{H^m}< C\de_1(N_1^{m-\frac12}+1)\,.
\end{equation}
Hence it follows from Corollary~\ref{C.stability} that if
\begin{equation}\label{ineq1}
\de_1 N_1^{r+\frac12}\ll N_0^{-r}
\end{equation}
then the solution~$u$
is globally defined and the difference
\[
v(x,t):=u(x,t)-w(x,t)
\]
is bounded as
\begin{equation}\label{vHm}
\|v(\cdot,t)\|_{H^m}\leq C( N_0^{m-1}+1)\e^{-\nu\si t}\|v_0\|_{H^m}\,,
\end{equation}
for all $0\leq m\leq r+1$, where the implicit constant in this inequality depends on~$M$. Here and in what follows, $r$ is a fixed
integer, which we can assume to be larger than or equal to 7 as in Lemma~\ref{L.contractible}.

We shall need more estimates for the difference $v$. To derive them,
notice that $v$ satisfies the equations
\[
\partial_{t} v + \Div( v \otimes  v + 
 2 v \otimes w) 
-   \nu\, \Delta  v  =  -  \nabla P_{ v}   \,,\quad \Div  v  =  0\,,\quad  v(\cdot,0) =   v_0\,.
\]
It is then standard that the field~$v$ can then be written, using Duhamel's
formula, as
\begin{equation}\label{v(t)}
v(\cdot, t)=e^{\nu t\De}v_0- 2\, \Lin(\cdot,t)-\Bil(\cdot,t)\,,
\end{equation}
where
\begin{equation}\label{v0(t)}
e^{\nu t\De}v_0= \sum_{j=1}^n\de_j\, \cW_j\, e^{-\nu N_j^2 t}\,,
\end{equation}
and the linear and bilinear terms are
\begin{align*}
\Lin(\cdot,t)& := \int_{0}^{t} e^{\nu (t-s)\Delta} \, \PP\, \Div
(v(s) \otimes w(s) )  \, ds \\
\Bil(\cdot,t)& := \int_{0}^{t} e^{\nu (t-s)\Delta} \, \PP\, \Div
(v(s) \otimes v(s) )  \, ds\,.
\end{align*}
Here we are writing $v(t)\equiv v(\cdot,t)$ for the ease of notation and $\PP$ denotes the Leray projector onto divergence-free fields
with zero mean, which is given by the Fourier multiplier 
\[
\widehat{\PP V}(k):=\frac{|k|^2I-k\otimes k}{|k|^2}\widehat V(k)\,,
\]
which is understood as zero on the zero mode.

Equation~\eqref{v(t)} shows that the difference $v(x,t)$ consists of
terms that evolve according to the heat equation and of two terms
whose evolution is more involved. To deal with these terms we will utilize
the following simple bounds for the heat kernel:

\begin{lemma}\label{L.heat}
Let $f$ be a function (or vector field) on $\TT^3$ with zero mean. For any integer $m\geq0$ and $s>0$ one has
\begin{align}\label{heat1}
\|e^{ s\De}f\|_{H^m}&\leq Cs^{-\frac m2}\|f\|_{L^2}\,,\\
\|e^{ s\De} f\|_{H^m}&\leq e^{-s}\|f\|_{H^m}\,. \label{heat2}
\end{align}
\end{lemma}
\begin{proof}
For the first bound we use that, as $f$ has zero mean,
\begin{align*}
\|e^{ s\De}f\|_{H^m}^2&=\sum_{k\in\ZZ^3\backslash\{0\}}
                        \frac{|k|^{2m}}{e^{s|k|^2}}|\hf(k)|^2
                        \leq Cs^{-m}\sum_{k\in\ZZ^3\backslash\{0\}}|\hf(k)|^2=Cs^{-m}\|f\|_{L^2}^2\,,
\end{align*}
where we have used Parseval's identity and the fact that
$e^z=\sum_{n=0}^\infty \frac{z^n}{n!}\geq \frac{z^m}{m!}$ for any
integer $m$ and any $z>0$. Likewise,
\begin{align*}
\|e^{ s\De}f\|_{H^m}^2&=\sum_{k\in\ZZ^3\backslash\{0\}}
                        \frac{|k|^{2m}}{e^{s|k|^2}}|\hf(k)|^2
                        \leq e^{-s}\sum_{k\in\ZZ^3\backslash\{0\}}|k|^{2m}|\hf(k)|^2=e^{-s}\|f\|_{H^m}^2\,,
\end{align*}
where we have used that $|k|\geq1$ for all $k\in\ZZ^3\backslash\{0\}$.
\end{proof}

In the following lemma we apply these estimates to obtain bounds for
the linear and bilinear terms in the Navier--Stokes evolution with the
right dependence on the free parameters $\de_k$, $N_k$. Here and in
what follows, all the implicit constants in expressions of the form $
N_0\gg1$ and in the various bounds that we derive depend on
the fixed parameters
$\nu,M,T_k$, but not on the free parameters $\de_k$,~$N_k$.

\begin{lemma}\label{L.bound}
For any $t\geq T_1$ and $N_0\gg N_1^{\frac{r-1}2}$, the above linear and bilinear terms are bounded as
\begin{align*}
\|\Lin(\cdot,t)\|_{H^r}&\leq C\de_1 N_0^{-2}  \,, \\
\|\Bil(\cdot,t)\|_{H^r}&\leq C\de_1^2 N_0^{r+1} N_1^{r+2}\,.
\end{align*}
\end{lemma}

\begin{proof}
Using the bounds~\eqref{v0} and~\eqref{vHm} for $\|v_0\|_{H^m}$ and $\|v\|_{H^m}$, a standard bilinear estimate and the Sobolev embedding, one can readily obtain
\begin{align*}
\|v(s)\otimes v(s)\|_{H^{r+1}}&\leq
C\|v(s)\|_{L^\infty}\|v(s)\|_{H^{r+1}}\leq
C\|v(s)\|_{H^2}\|v(s)\|_{H^{r+1}}\\
&\leq C\de_1^2 N_0^{r+1} N_1^{r+2}\, e^{-2\nu\si s}\,,\\
\|v(s)\otimes w(s)\|_{L^2}&\leq
C\|w(s)\|_{L^\infty}\|v(s)\|_{L^2}\leq C\de_1
 e^{-\nu N_0^2 s}\,,\\[1mm]
\|v(s)\otimes w(s)\|_{H^{r+1}}&\leq
C\|w(s)\|_{L^\infty}\|v(s)\|_{H^{r+1}}+C\|v(s)\|_{L^\infty}\|w(s)\|_{H^{r+1}}\\
&\leq
C\|w(s)\|_{L^\infty}\|v(s)\|_{H^{r+1}}+C\|v(s)\|_{H^2}\|w(s)\|_{H^{r+1}}\\
&\leq
C\de_1( N_0^r N_1^{r+\frac12}+ N_0^{r+2} N_1^{\frac32})\,
e^{-\nu N_0^2s}\\
&\leq
C\de_1 N_0^{r+2} N_1^{\frac32}\, e^{-\nu N_0^2s}\,,
\end{align*}
where in the last bound we are assuming that $N_0\gg N_1^{\frac
  {r-1}2}$.

One can now use these bounds with the heat kernel estimates stated in
Lemma~\ref{L.heat} to estimate the linear term as
\begin{align*}
\|\Lin(\cdot,t)\|_{H^r}&\leq \int_{0}^{t} \|e^{\nu (t-s)\Delta} \, \PP\, \Div
(v(s) \otimes w(s) ) \|_{H^r} \, ds\\
&\leq C\int_{0}^{t/2} \|e^{\nu (t-s)\Delta} \, \PP
(v(s) \otimes w(s) ) \|_{H^{r+1}} \, ds \\
&\qquad \qquad \qquad \qquad\qquad\qquad +  C\int_{t/2}^{t} \|e^{\nu (t-s)\Delta} \, \PP
(v(s) \otimes w(s) ) \|_{H^{r+1}} \, ds\\
&\leq C\int_{0}^{t/2} (t-s)^{-\frac{r+1}2}\|v(s) \otimes w(s)  \|_{L^2} \, ds \\
&\qquad \qquad \qquad \qquad\qquad\qquad+ C\int_{t/2}^{t}
e^{-\nu(t-s)}\|v(s) \otimes w(s)  \|_{H^{r+1}} \, ds\\
&\leq C\de_1
 \int_{0}^{t/2} (t-s)^{-\frac{r+1}2} e^{-\nu N_0^2 s}\, ds \\
&\qquad \qquad \qquad \qquad\qquad\qquad+ C\de_1 N_0^{r+2} N_1^{\frac32}\,\int_{t/2}^{t} e^{-\nu(t-s)}
e^{-\nu N_0^2s}\, ds\\
&\leq C\de_1 N_0^{-2}
 + C\de_1 N_0^{r} N_1^{\frac32}e^{-\nu N_0^2t/2}\\
&<C\de_1 N_0^{-2}
\end{align*}
for $N_0\gg1$ and $t\geq T_1$. Likewise,
\begin{align*}
\|\Bil(\cdot,t)\|_{H^r}&\leq C\int_{0}^{t} \|e^{\nu (t-s)\Delta} \, \PP\, 
(v(s) \otimes v(s) ) \|_{H^{r+1}} \, ds\\
&\leq C \int_{0}^{t} e^{-\nu (t-s)} \, \|v(s) \otimes v(s)  \|_{H^{r+1}}
\, ds\\
& \leq C\de_1^2 N_0^{r+1} N_1^{r+2}\,.
\end{align*}
\end{proof}

\subsubsection*{Step 3: Choice of the constants and conclusion of the
  proof}

Our goal now is to choose the constants $\de_k, N_k$ so that the
vortex structures of the fluid at time~$T_k$ are those of the field
$\cW_k$, modulo a small deformation that does not change their
topology. In order to do that, let us take a look at the
formula~\eqref{v(t)} for the field $v(\cdot,t)$. 

The strategy is to
ensure that the terms $\Lin(\cdot,t)$ and $\Bil(\cdot,t)$ (which are
harder to analyze) do not significantly contribute to the solution for
times up to $T_n$. This enables us to essentially restrict our
attention to the simple term $e^{\nu t\De} v_0$, where we will choose
the constants so that the leading term at time $T_k$ is precisely
$\de_k e^{-\nu N_k^2 T_k}\cW_k$. To make things precise and derive
effective estimates we will need to rescale the field at each time $T_k$.

To implement this strategy, let us start with the time $T_0=0$. Since 
\[
\|v_0\|_{H^r}\leq C \de_1  N_1^{r-\frac12}
\]
by the bound~\eqref{v0} and
taking into account that~$\cW_0=B_{N_0}$, it is clear
that if
\begin{equation}\label{ineq2}
\de_1  N_1^{r-\frac12} \ll  1\,,
\end{equation}
then Lemma~\ref{L.contractible} ensures that the field
\[
\tu_0:=M^{-1}u_0=\cW_0+M^{-1}v_0
\]
does not have any contractible vortex lines or vortex tubes. As
$M^{-1}u_0$ is simply a rescaling of the solution~$u$ at time~$0$, we
infer that if the above condition is satisfied, then all the vortex
lines and tubes of the fluid at time~0 are non-contractible. By
Lemma~\ref{L.contractible}, this property is structurally stable (and
is therefore satisfied for all small enough times).

Let us next consider the behavior of the fluid at time~$T_k$, with $1\leq k\leq
n$. Here it is convenient to rescale the velocity field by defining
\[
\tu_k:=\de_k^{-1}\, e^{\nu  N_k^2 T_k}\, u(\cdot, T_k)\,.
\]
It is clear that
\begin{multline*}
\tu_k=\cW_k+\frac M {\de_k}\, e^{-\nu ( N_0^2- N_k^2) T_k}\,\cW_0+
\sum_{1\leq j\neq k\leq n} \frac{\de_j}{\de_k}\, e^{\nu ( N_k^2
  - N_j^2) T_k} \cW_j\\
+ \de_k^{-1}\, e^{\nu  N_k^2 T_k}\,\big(\Lin(\cdot,T_k)+\Bil(\cdot,T_k)\big)\,.
\end{multline*}
Our objetive is to choose the constants $\de_j, N_j$ so that, for all
$1\leq k\leq n$,
\begin{align}
\frac{\de_{k-1}}{\de_k}\, e^{-\nu(N_{k-1}^2-N_k^2)T_k}&\ll
N_{k-1}^{-r}\,,\label{ineq3}\\
\frac{\de_{k+1}}{\de_k}\, e^{\nu(N_k^2-N_{k+1}^2)T_k}&\ll N_{k+1}^{-r}\,,\label{ineq4}
\end{align}
where $\de_0:=M$ and the second condition is of course
absent for $k=n$. As $N_k\gg
N_{k+1}$, $\de_k\gg\de_{k+1}$ and $T_k< T_{k+1}$, it is not
hard to see that if these conditions are satisfied for all $1\leq
k\leq n$, then
\begin{equation}\label{good1}
\bigg\| \frac M {\de_k}\, e^{-\nu ( N_0^2- N_k^2) T_k}\,\cW_0+
\sum_{1\leq j\neq k\leq n} \frac{\de_j}{\de_k}\, e^{\nu ( N_k^2
  - N_j^2) T_k} \cW_j\bigg\|_{H^r}\ll 1\,.
\end{equation}

We will also impose that 
\[
\de_k^{-1}\, e^{\nu  N_k^2
  T_k}\,\big(\|\Lin(\cdot,T_k)\|_{H^r}+\|\Bil(\cdot,T_k)\|_{H^r}\big)\ll 1
\]
for all $1\leq k\leq n$. Since
\[
\de_k^{-1}\, e^{\nu  N_k^2
  T_k}\ll   \de_{k+1}^{-1}\, e^{\nu  N_{k+1}^2
  T_k} \ll \de_{k+1}^{-1}\, e^{\nu  N_{k+1}^2
  T_{k+1}}\,,
\]
as a consequence of~\eqref{ineq3}-\eqref{ineq4}, it follows that it
suffices to impose that
\begin{equation}\label{ineq5}
\de_n^{-1}\, e^{\nu  N_n^2
  T_n}\,(\de_1N_0^{-2}+ \de_1^2N_0^{r+1}N_1^{r+2})\ll 1\,,
\end{equation}
where we have used the bounds for $\Lin$ and $\Bil$ derived in
Lemma~\ref{L.bound}.

If $1\leq k\leq n$ is odd, it then follows from Lemma~\ref{L.contractible}
that if the above conditions for the constants are satisfied, the
field $\tu_k$ (which is just a rescaling of $u(\cdot,T_k)$) has a
collection of vortex tubes and vortex lines diffeomorphic to the
set~$\cS_k$, and that this set is structurally stable. Likewise, if
$k$ is even, if the conditions are satisfied
Lemma~\ref{L.contractible} ensures that all the vortex lines and
vortex tubes of $\tu_k$ (and therefore of $u(\cdot,T_k)$) are
non-contractible, and that this property is structurally stable. This
automatically yields Theorem~\ref{T1}.

Hence it only remains to show that one can indeed choose the constants
$N_k$, $\de_k$ so that the conditions~\eqref{ineq1},
\eqref{ineq2}--\eqref{ineq4} and~\eqref{ineq5} are satisfied. To prove
this, starting with $N_n$, let us take any large constants
\[
1\ll N_n\ll N_{n-1}\ll\cdots \ll N_2 \ll \sqrt{N_1}\,,
\]
and set, for $1\leq k\leq n-1$,
\[
\rho_k:=e^{-\nu N_k^2(T_k+T_{k+1})/2}\,.
\]
In view of the first term in the
inequality~\eqref{ineq5}, let us take $N_0\gg N_1$ so that
\[
N_0^2\gg N_1^{r-1}\, e^{\nu T_n N_n^2}\prod_{k=1}^{n-1}\rho_k^{-1}\,.
\]
Notice that this implies that $N_0\gg N_1^{\frac{r-1}2}$, as assumed in
  Lemma~\ref{L.bound}.

In view of the second term in~\eqref{ineq5}, we now set
\[
\de_1:=c \,N_0^{-r-1}\,N_1^{-r-2}\, e^{-\nu T_n N_n^2}\prod_{k=1}^{n-1}\rho_k\,,
\]
where $c\ll 1$ is a small positive constant that does not depend on
the frequencies.
A short computation now shows that all the above conditions are
satisfied if one now recursively sets, for $1\leq k\leq n-1$,
\[
\de_{k+1}:=\de_k\,\rho_k\,.
\]
Notice that one then has $\prod_{j=1}^{k-1}\rho_k=\de_k/\de_1$.

Theorem~\ref{T1} is then proved, up to the minor issue that the $L^2$
norm of $u_0$ in the above construction is not $M$ but
\[
\|u_0\|_{L^2}=M+O(\de_1)\,.
\]
The final result then follows upon replacing the initial condition $u_0$ by
$q\, u_0$, with $q:=M/\|u_0\|_{L^2}$ satisfying $|q-1|< C\de_1$. It is
apparent that this factor does not change anything in the above
arguments, so one readily obtains the desired result.

\begin{remark}\label{R.Reynolds}
One can prove the same result where, instead of imposing that
the $L^2$~norm of the initial datum is an arbitrary constant~$M$, one
imposes that its Reynolds number is~$M$. We recall that the Reynolds
number is usually defined as
\[
\cR(u_0):=\frac{\|(u_0\cdot \nabla) u_0\|_{L^2}}{\nu\|\De u_0\|_{L^2}}\,.
\]
The proof is exactly as above, the only difference being that the initial
condition that one must take is
\[
u_0:=\nu M N_0\, \widetilde B_{N_0}+ \sum_{k=1}^n \de_k\,\cW_k\,,
\]
where the constants $N_k, \de_k$ can be chosen as before and we are
considering the family of Beltrami fields
\[
\widetilde B_N:=\sqrt 5\,(2\sin Nx_3,\sin Nx_1+2\cos Nx_3,\cos Nx_1)\,,
\]
which satisfy $\curl \widetilde B_N=N\widetilde B_N$.
This is the right normalization constant for~$\widetilde B_N$ in
this context, as it ensures that
\[
\frac{\|(\widetilde B_N\cdot \nabla) \widetilde B_N\|_{L^2}}{\|\De \widetilde B_N\|_{L^2}}=\frac1N\,,
\]
which easily implies that $\cR(u_0)$ is essentially~$M$.

Two important observations are in order. The first one is that the
factor $N_0^\ga$ with $\ga=1$ that we have put in front of $\widetilde
B_{N_0}$ does not change anything in the argument, while for $\ga>1$
it would change it dramatically. The reason is that the $L^2L^\infty$
norm of $N_0^\ga\, \widetilde B_{N_0}\, e^{-\nu N_0^2 t}$, which
appears in an exponential in Theorem~\ref{T.stability}, is uniformly
bounded with respect to the parameter~$N_0$ precisely for
$\ga\leq1$. The second observation is that we have replaced $B_{N_0}$
by $\widetilde B_{N_0}$ because $B_{N_0}$ satisfies the algebraic
identity
\[
(B_{N_0}\cdot\nabla)B_{N_0}=0\,,
\]
and this would have made the Reynolds number artificially small. It is worth
mentioning that the family $\widetilde B_N$ enjoys the same property
of robust non-contractibility established for the family $B_N$ in
Lemma~\ref{L.contractible}, and that the proof follows the same lines
(the algebra gets a little more awkward, though).
\end{remark}

\section{Instantaneous destruction of vortex tubes}
\label{S.destruction}

In this section we will prove Theorem~\ref{T2}. Given a $C^\infty$ function $h:\TT^2\to\RR$, let us consider as initial
vorticity the field
\begin{multline*}
\om_0:=M\,\big(\sin (x_3+\ep\, h), \cos (x_3+\ep\,
h),\\
 -\ep\pd_1h\, \sin (x_3+\ep\, h)-\ep\pd_2h\,\cos (x_3+\ep\, h)\big)\,,
\end{multline*}
where $M$ is a constant (which will be thought of as large), $h\equiv h(x_1,x_2)$ and $\ep$ is a small positive constant that will be specified
later. Notice that $\om_0$ is divergence-free and of zero mean, and that the associated
initial velocity can be written as
\begin{equation}\label{u0}
u_0:=M\,\big( \sin (x_3+\ep\, h),\cos (x_3+\ep\, h),0\big)+ \ep M\, \nabla\phi\,,
\end{equation}
where $\phi$ is the only solution to the elliptic equation on $\TT^3$
\[
\De\phi=\pd_2h\, \sin (x_3+\ep\, h) -\pd_1 h\, \cos (x_3+\ep\, h)\,,\qquad \int\T\phi\, dx=0\,.
\]
Notice that $u_0$ is then divergence-free and of zero mean.

Observe that the difference between the initial velocity and the
Beltrami field of unit frequency $W:=M(\sin x_3,\cos x_3,0)$ is obviously bounded as
\[
\|u_0-W\|_{H^k}< CM\ep\,.
\]
Since the solution to the Navier--Stokes equations with datum~$W$ is
$w(x,t)=e^{-\nu t}W(x)$, Theorem~\ref{T.stability} ensures that if
$M\ep$ is small enough there is a global smooth solution $u(x,t)$ to the
Navier--Stokes equations with initial datum~$u_0$.

To study the structure of the vortex tubes at time~0, it is convenient
to start by noticing that $\om_0$ is conjugated to the above Beltrami
field~$W$. Specifically, given the volume-preserving diffeomorphism of $\TT^3$
\begin{equation}\label{Phi}
\Phi(x):=(x_1,x_2,x_3+\ep\,h(x_1,x_2))\,,
\end{equation}
it can be readily checked that $\om_0$ can be written as the pullback
of~$W$ by~$\Phi$:
\[
\om_0=\Phi^*W\,.
\]
Since the surfaces $x_3=\text{constant}$ are vortex tubes for the
Beltrami field~$W$ and the vortex lines on these tubes are periodic or
quasi-periodic depending on whether the number~$\tan x_3$ is rational
or irrational, it stems that the same picture is valid for~$\om_0$, so
at time zero the torus $\TT^3$ is covered by vortex tubes with
periodic or quasi-periodic vortex lines.

We shall next use Melnikov's theory to show that some of these vortex
tubes (the ones covered by periodic vortex lines, which are resonant
invariant tori of the vorticity) can break down instantaneously. For
this we will need to consider the evolution of the equation for the
vortex lines. The vorticity formulation for the Navier--Stokes
equations ensures that
\begin{equation}\label{om}
\om=\om_0+t\,\big(\nu\,\De \om_0+(\om_0\cdot\nabla)
u_0-(u_0\cdot\nabla )\om_0\big) + O(t^2)\,.
\end{equation}
To study the vortex lines we introduce coordinates on $\TT^3$ associated
with the diffeomorphism~\eqref{Phi}, which we denote by
\[
X\equiv (X_1,X_2,X_3):=(x_1,x_2,x_3+\ep\, h(x_1,x_2))\,.
\]
In terms of these coordinates, the ODE for the vortex lines,
\[
\frac {d x}{d\tau}=\om(x,t)\,,
\]
where $\om(x,t)$ is given by~\eqref{om} and~$\tau$ is the parameter of
the vortex lines, reads as
\begin{equation}\label{eqX}
\frac{ dX}{d\tau}=M\,(\sin X_3,\cos X_3,0)+ t \, F+O(t^2)\,.
\end{equation}
The components of the vector field~$F$ are obtained by changing variables in
Eq.~\eqref{om}. After expanding in $\ep$ and performing a
straightforward but tedious computation one obtains that the third
component of $F$ can be written as
\begin{multline}\label{F3}
F_3=2\nu M\ep^2\big( \sin X_3  \,\pd_2h\,\pd_{22} h+ \sin X_3  \,\pd_1h\,\pd_{12} h-  \cos X_3 \,
 \pd_2  h\, \pd_{12} h- \cos X_3  \,\pd_1h\,
   \pd_{11}h\big) \\
+ (\sin X_3\,\pd_1+\cos X_3\,\pd_2)\psi+O(\ep^3)\,,
\end{multline}
where $\psi$ is certain function on~$\TT^3$ of order $O(\ep)$ 
whose (rather awkward) 
expression will not be needed. Here the partial derivatives should be
interpreted as
\[
\pd^\al f \equiv (\pd^\al f)(X_1,X_2,X_3-\ep \, h(X_1,X_2))\,,
\]
that is, as the composition of the function $\pd_x^\al f(x)$ (the
derivatives being taken with respect to the original variables~$x$)
with the diffeomorphism $\Phi^{-1}$ that passes from the
coordinates~$X$ to~$x$.  Notice that the term $O(\ep^3)$ is not uniformly
bounded in~$M$.

To apply Melnikov's theory we shall write the ODE for the vortex lines
as a non-autonomous dynamical system on the plane. To this end we
shall restrict our attention to a region covered by vortex tubes of
the initial vorticity and where $\sin X_3$ is nonnegative, such as
\begin{equation}\label{region}
X_3\in\Big(\frac\pi4,\frac{3\pi}8\Big)\,,
\end{equation}
and to small times, which is not a serious drawback as we intend to
prove the instantaneous destruction of invariant tori. We shall then
define, in this region and for small enough~$t$, a new parameter for the vortex lines as
\[
ds:=(M\, \sin X_3(\tau)+t \, F_1(X(\tau),t) + O(t^2))\, d\tau\, \,,
\]
so that Equation~\eqref{eqX} can be written as
\begin{subequations}\label{planar}
\begin{align}
\frac{dX_1}{ds}&=1\,,\\
\frac{d X_2}{ds}&= \cot X_3+O(t)\,,\\
\frac{d X_3}{ds}&=\frac{t\, F_3}{M\,\sin X_3}+O(t^2)\,.
\end{align}
\end{subequations}

Since one can integrate the first equation to find that
\[
X_1=s+\xi \,,
\]
where $\xi$ is a constant, it is now enough to analyze the nonautonomous planar system defined by
the second and third components of the above system with $X_1$
replaced by $s+\xi $. At $t=0$, the integral curves of this planar
system are
\begin{equation}\label{intcurvX}
X_2=X_2^0 +s\,\cot X_3^0\,,\qquad X_3= X_3^0\,,
\end{equation}
so, since $(X_2,X_3)\in\TT^2$, it turns out that the planar integral curve
with initial condition $(X_2^0,X_3^0)$ is periodic if and only if
$\cot X_3^0$ is a rational number, its period being~$2\pi q$ if $\cot
X_3^0$ is given by the irreducible fraction $p/q$.

Let us consider a vortex tube of the initial vorticity given in the
coordinates~$X$ by the equation
\[
\cot X_3=\frac p q\,,
\]
with $p$, $q$ coprime integers and $p/q$ in the interval
$(\cot\frac{3\pi}8,1)$, the latter restriction coming from the inclusion~\eqref{region}. In terms of the associated nonautonomous
planar system, the corresponding Melnikov function is~\cite[Theorem 4.6.2]{GH}
\[
\cM(\xi ):=\int_0^{2\pi q}\frac{\cot X_3\, F_3}{M\,\sin
  X_3}\bigg|_{(s+\xi , X_2^0+\frac pq s ,\arctan \frac qp)}\, ds
\]
It is standard that $\cM(\xi)$ does not depend on~$X_2^0$, so we will
take $X_2^0:=0$. Using the
expression~\eqref{F3} for~$F_3$, one immediately finds that
\begin{multline*}
\cM(\xi )=\frac{2p\nu \ep^2}q\int_0^{2\pi q}\Big( \pd_2h\,\pd_{22} h+
\pd_1h\,\pd_{12} h-  \frac pq \,
 \pd_2  h\, \pd_{12} h- \frac pq  \,\pd_1h\,
   \pd_{11}h\Big)\Big|_{(s+\xi , \frac pq s ,\arctan \frac qp)}\, ds\\
+\frac{p}{qM}\int_0^{2\pi q}\Big(\pd_1+\frac
pq\,\pd_2\Big)\psi\big|_{(s+\xi , \frac pqs ,\arctan\frac qp)}\, ds+ O(\ep^3)\,.
\end{multline*}
It is clear that the second integral vanishes by periodicity because
\[
\int_0^{2\pi q}\Big(\pd_1+\frac
pq\,\pd_2\Big)\psi\big|_{(s+\xi , \frac pqs ,\arctan\frac qp)}\, ds=
\int_0^{2\pi q}\frac d{ds}\Big(\psi\big|_{(s+\xi , \frac pqs ,\arctan\frac qp)}\Big)\, ds=0\,.
\]

Our goal now is to choose the function~$h$ so that the Melnikov
function vanishes and all its zeros are simple. It is not hard to see
that there are many ways to accomplish this. For example, one can make
the computations with the choice
\begin{equation}\label{hcos}
h(x_1,x_2):=\cos(p x_1-q x_2)
\end{equation}
to obtain
\begin{equation}\label{cM}
\cM(\xi)= -\frac{2\pi\nu \ep^2 p(p^2+q^2)^2}q\sin(2p\xi)+O(\ep^3)\,,
\end{equation}

Since the vorticity is divergence-free and the diffeomorphism $\Phi$
is volume preserving, it is clear that the flow of~\eqref{eqX}
preserves the measure $dX:= dX_1\, dX_2\, dX_3$. In turn, this implies
that the flow of the rescaled system~\eqref{planar} preserves the measure
\[
\rho\, dX\,,
\]
with
\[
\rho:= M\, \sin X_3+t \, F_1(X,t) + O(t^2)\,,
\]
which is well defined in the region~\eqref{region} for small
enough~$t$. 

Let us now show that the vector field~$V$ defined by the right hand
side of the rescaled system~\eqref{planar} satisfies two
important additional technical conditions. Firstly, the 2-form defined
by
\[
\al:= i_V(\rho\, dX_1\wedge dX_2\wedge dX_3)
\]
where $i_V$ denotes the inner product of the vector field~$V$ with a 3-form, is exact. In order
to see this, notice that $\rho\, V=\Phi_*\om$ is the expression of the
vorticity in the coordinates~$X$, so one has
\begin{align*}
\al= i_{\rho V}(dX_1\wedge dX_2\wedge dX_3)=\Phi_*\big( i_\om
(dx_1\wedge dx_2\wedge dx_3)\big)\,.
\end{align*}
Let us denote by $\be$ the 1-form dual to the velocity field~$u$, defined
in terms of the Euclidean metric as
\[
\be(v):=u\cdot v\,.
\]
Using differential forms to characterize $\om:=\curl u$, it
is standard that
\[
i_\om (dx_1\wedge dx_2\wedge dx_3)=d\be\,,
\]
which shows that $\al$ is exact, as claimed:
\[
\al= \Phi_* (d\be)=d(\Phi_*\be)\,.
\]

The second technical fact is that the vector field~$V$ satisfies a
twist condition. More precisely, by Equation~\eqref{intcurvX} the period of the function $X_2(s)$ is $2\pi/\cot
X_3^0$, whose derivative with respect to~$X_3^0$ does not vanish in
the interval~\eqref{region}, so the period of the vortex lines in the variable~$s$ is
different, in general, on distinct vortex tubes.

Since these technical conditions are satisfied, one can then apply a
Melnikov-type theorem~\cite[Theorem 4.8.3]{GH} to the
function~\eqref{cM}. The function $\cM(\xi)$ has exactly 4~zeros in the
interval $\xi\in [0,2\pi/p)$, all of which are nondegenerate (meaning
that the derivative does not vanish at these points). The theorem then
ensures that the invariant torus of equation $\cot X_3=p/q$ breaks
down for all small enough positive times, and that only 4~among the
period~$2\pi q$ integral curves on this invariant torus survive for
small positive times;
furthermore, two of them are elliptic and the other two are
hyperbolic.

Since $\rho V=\Phi_*\om$, the integral curves of the field~$V$ are
diffeomorphic to those of the vorticity, so from the above statement
about the breakdown of the invariant tori of~$V$ it stems that one of
the vortex tubes at initial time (corresponding to $\Phi^{-1}(\{X:
\cot X_3=p/q\})$) also breaks down instantaneously.

Theorem~\ref{T2} is then proved. More precisely, the Melnikov
theory~\cite[Theorem 4.8.3]{GH} guarantees that, as a consequence of  the instantaneous breakdown of the invariant torus $\cot X_3=p/q$, at any small
  enough positive times there appear integral curves (homoclinic or heteroclinic connections) that are not periodic or quasiperiodic and are
  not tangent to an invariant torus. The Melnikov theory also ensures
  that, by bifurcation, for any small positive time the destruction of
  this invariant torus gives rise to two elliptic periodic integral
  curves and to two hyperbolic periodic vortex lines with intersecting stable
  and unstable manifolds.

\begin{remark}\label{R.nonlinear}
It follows from the proof of the theorem that, in fact, the
contribution of the nonlinear term $(\om_0\cdot\nabla)
u_0-(u_0\cdot\nabla )\om_0$ to the Melnikov function is zero, as one
can infer from the fact that the Euler equation does not feature
vortex reconnection. This is a general fact, and does not depend on
our choice of the initial datum. Notice, moreover, that
Equation~\eqref{cM} ensures that the
size of the perturbation, $\ep$, must be of order~$o(\nu)$, which
explains why this scenario of instantaneous vortex reconnection does
not survive in the vanishing viscosity limit.
\end{remark}

\begin{remark}
  A very minor modification of the argument permits to break
  instantaneously any finite number of the initial configuration of
  vortex tubes, not just one.
\end{remark}

\section{Conclusions}
\label{S.discussion}

In this section we shall collect a number of remarks and observations
about the proofs of Theorems~\ref{T1} and~\ref{T2} that provide
further insight into this scenario of vortex reconnection.

Firstly, notice that the PDE aspects of the proofs of
Theorems~\ref{T1} and~\ref{T2} are essentially linear, in the sense
that we are concerned with small perturbations of a solution to the
Navier--Stokes equations of the form $w(x,t):=M\, e^{-\nu N_0^2 t}\,
B_{N_0}(x)$ with $B_{N_0}$ a Beltrami field of high frequency~$N_0$. Notice
that, as vortex reconnection is a dissipative effect, it should not be too
surprising that the proof can be carried out in an essentially linear regime. This is seen
very clearly in Remark~\ref{R.nonlinear}. The way this should be
interpreted is that, although the evolution of the fluid takes place,
in general, in the nonlinear regime, in the creation or destruction of
vortex structures the straw that actually breaks the laden camel's
back is in fact essentially linear. 

The heart of the proof is a high-frequency analysis
in which we study which one among several terms of different frequencies (all of
which are large) is dominant at different time scales $T_0<T_1<\cdots
<T_n$. Although this analysis is made simpler and finer by the fact
that the terms can be taken as Beltrami fields, the underlying
interplay between different frequencies could have been carried out in
more general situations. Notice that, for any choice of the
distinct-frequency terms, to rigorously pass from the frequency analysis to the
statement that there is indeed the change of topology that constitutes
the vortex reconnection, it is essential to have a ``stable
topological non-equivalence'' theorem like our
Lemma~\ref{L.contractible}. This result can be extended to cover more
general families of vector fields, but in all cases it is a quite non-trivial KAM-theoretic argument in
itself.

The proof of the aforementioned Lemma~\ref{L.contractible} (and
therefore that of Theorem~\ref{T1}) uses in a crucial way the periodic
boundary conditions, that is, the fact that the spatial variable takes
values in~$\TT^3$. This is because this allows us to play with the
quite robust concept of contractibility. In contrast, this condition
does not play a role in the proof of Theorem~\ref{T2}, and in fact one
can establish a similar result on~$\RR^3$ with a completely analogous
reasoning.

It is also worth mentioning that the reconnection mechanism that we have presented in this paper does
not depend much on the form of the dissipative term of the
Navier--Stokes equations. Indeed, the fact that the dissipative term
is given by the Laplacian has only been used to write that the time-dependent
factor that appears in solutions to the Navier--Stokes
equations whose initial datum is a Beltrami field goes as the
exponential of the squared frequency, which is not essential, and to
derive heat kernel estimates. In particular, the argument goes through
for analogs of the Navier--Stokes equations that feature fractional
dissipation of the form
\[
\pd_t u+ (u\cdot\nabla)u+\nu(-\De)^\al u=-\nabla P\,,\qquad \Div u=0\,,\qquad
u(\cdot,0)=u_0
\]
with $\al>0$.

The fact that unstable vortex structures such as resonant tori can
break down instantaneously explains why there is abundant literature
on ``possibly robust'' vortex structures, such as vortex lines and
vortex tubes (which are robust when they satisfy suitable
non-degeneracy conditions), but not on, for example, ``vortex balls''
or ``vortex pretzels'' (that is, spheres or genus-2 tori consisting of
vortex lines), for which no robustness properties are expected in the
theory of divergence-free dynamical systems.

To conclude, it is worth explaining why the
difference between two fields needs to be controlled in several parts
of the proofs through the norm $\|\curl W-\curl W'\|_{C^{3,\al}}$ (or,
for convenience, through the stronger norms $\|W-W'\|_{C^{4,\al}}$ or
$\|W-W'\|_{H^7}$). The reason is that controlling the vortex tubes of
$W'$ in terms of those of $W$ essentially boils down to controlling
invariant circles of suitable annulus diffeomorphisms defined by the
flow of the fields $\curl W$ and $\curl W'$, and in this context it is
known that $C^{3,\al}$~bounds are sufficient~\cite{Herman} (and
essentially necessary~\cite{Cheng}) for the convergence of a KAM scheme.

\section*{Acknowledgments}

The authors are supported by the ERC Starting Grants~633152 (A.E.),
277778 (R.L.) and~335079
(D.P.-S.). This work is supported in part by the
ICMAT--Severo Ochoa grant
SEV-2015-0554.

\bibliographystyle{amsplain}

\end{document}